\documentclass[12pt]{amsart}

\usepackage{amsmath}
\usepackage{amsfonts}
\usepackage{amssymb}
\usepackage{amsthm}
\usepackage{hyperref}
\usepackage{enumerate}
\usepackage{cleveref}
\usepackage{mathrsfs}
\usepackage[top=0.9in, bottom=1.15in, left=1.1in, right=1.1in]{geometry}
 \usepackage{hyperref}
\hypersetup{colorlinks=true,linkcolor=blue,citecolor=magenta}
\newtheorem{theorem}{Theorem}

\newtheorem{proposition}[theorem]{Proposition}
\newtheorem{corollary}[theorem]{Corollary}

\newtheorem*{remark}{Remark}

\Crefname{conjecture}{Conjecture}{Conjectures}

\theoremstyle{plain}

\theoremstyle{plain}

\author[Akande-Genao-Haag-Hendon-Pulagam-Schneider-Sills]{A. P. Akande, Tyler Genao, Summer Haag, Maurice D. Hendon, \\Neelima Pulagam, Robert Schneider and Andrew V. Sills}

\address{Department of Mathematics\newline
University of Georgia\newline
Athens, Georgia 30602, U.S.A.}
\email{agbolade.akande@uga.edu}
\email{tylergenao@uga.edu}
\email{summer.haag@uga.edu}
\email{mhendon@uga.edu}
\email{np26102@uga.edu}%
\address{Department of Mathematical Sciences\newline
Michigan Technological University\newline
Houghton, Michigan 49931, U.S.A.}
\email{robertsc@mtu.edu}

\address{Department of Mathematical Sciences\newline
Georgia Southern University\newline
Statesboro, Georgia 30458, U.S.A.}
\email{asills@georgiasouthern.edu}

\title{Computational study of non-unitary partitions}
 
\begin{document}
 
\begin{abstract}
Following Cayley,  MacMahon, and Sylvester, define a non-unitary partition to be an integer partition with no part equal to one, and let $\nu(n)$ denote the number of non-unitary partitions of size $n$. In a 2021 paper, the sixth author proved a formula to compute $p(n)$ by enumerating only non-unitary partitions of size $n$, and recorded a number of conjectures regarding the growth of $\nu(n)$ as $n\to \infty$. Here we refine and prove some of these conjectures. For example, we prove $p(n) \sim \nu(n)\sqrt{n/\zeta(2)}$ as $n\to \infty$, and give Ramanujan-like congruences between $p(n)$ and $\nu(n)$ such as $p(5n)\equiv \nu(5n)\  (\operatorname{mod} 5)$.  \end{abstract}

\maketitle

\section{Introduction and statement of results}

\subsection{Non-unitary partitions and $p(n)$}\label{sect1}

Let $\mathcal P$ denote the set of integer partitions, 
including the empty partition $\emptyset$. Let $\mathcal P_n$ denote partitions of size (sum) equal to $n\geq 0$, with $\mathcal P_0:=\{\emptyset\}$.  The {\it partition function} $p(n)$ gives the cardinality $\#\mathcal P_n$, with $p(0):=1$ \cite{Andrews_theory}. 

In a 2021 paper \cite{non-unitary}, the sixth author (Schneider) studies the class of {\it non-unitary partitions},\footnote{Non-unitary partitions are referred to as ``nuclear partitions'' in \cite{non-unitary}.} which are  partitions having no part equal to one, and records a number of conjectures. Non-unitary partitions appear to have first arisen in the literature in the 1880s in the work of MacMahon \cite{N2} and Sylvester \cite{N3}  in connection with seminvariants in the theory of invariants. Cayley made another early contribution with \cite{N1}.  Guy \cite{Guy} proved the number of non-unitary partitions of size $n$ into odd parts, equals the number of partitions of size $n$ into distinct parts none of which is a power of two. The present authors, using tools from number theory and computer science, made a computational study of non-unitary partitions that led us to prove some of the conjectures in \cite{non-unitary}. 



Let $\mathcal N\subset \mathcal P$ denote the set of non-unitary partitions, 
let $\mathcal N_n\subset \mathcal P_n$ denote non-unitary partitions of size $n\geq 0$, and let $\nu(n)$ denote the cardinality $\#\mathcal N_n$, with $\nu(0):=1$. 
It is not hard to see 
\begin{equation}\label{nu_recursion}
p(n)\  =\  p(n-1)+\nu(n),
\end{equation} 
since every partition of $n$ that is {\it not} non-unitary, can be obtained by adjoining 1 to a partition of $n-1$. Immediately this recursion implies 
$p(n)=\nu(0)+\nu(1)+\nu(2)+\cdots+\nu(n).$ 
As a result of \eqref{nu_recursion}, in \cite{non-unitary} it is proved one can view partitions in $\mathcal N_n$ as ``decaying'' to produce the rest of $\mathcal P_n$ by a certain algorithm resembling nuclear decay, which yields a formula for $p(n)$ requiring one to generate only {non-unitary} partitions of size $n$ (see \cite[Thm. 1]{non-unitary}). To see this is useful, let us compare $\#\mathcal N_n$ to $\#\mathcal P_n$.
%
%
In \cite{non-unitary} it is proved using the Hardy--Ramanujan asymptotic \cite[Eq. 1.41]{H-R}, viz. 
\begin{equation}\label{H-R}
p(n) \sim \frac{e^{A\sqrt{n}}}{Bn}\  \  \text{with} \  \  A=\pi\sqrt{2/3},\  \  B=4\sqrt{3},
\end{equation}
that $\nu(n)=o(p(n))$. What more can one deduce about the relative growths of $p(n), \nu(n)$?

\subsection{Our main results}
Based on numerical patterns in Table \ref{Table} below,\footnote{We note the $n=0$ row is omitted from Table \ref{Table}. In \cite{non-unitary}, the  initial values $\gamma(0)=\nu(0)=0$ are defined; alternatively, the definitions $\gamma(0)=\nu(0)=1$ are consistent with the convention  $p(0)=1$.} 
an explicit comparison relating $p(n)$ and $\nu(n)$ is conjectured in \cite{non-unitary}: 
\vskip.1cm \begin{center}{\it \label{conj1} As $n\to \infty$, we have 
$p(2n)\  \approx\  \nu(2n)\sqrt{2n},$}\end{center}
\vskip.1cm
where ``$\approx$'' means approximately equal in magnitude. However, the approximation  as stated 
is not compatible with \eqref{H-R}. Our more extensive computations 
up to $n=3000$ suggested the estimate needs a multiplicative constant, and extends to odd $n$ as well. 

Recall the well-known evaluation $\zeta(2)=\pi^2/6$ of the Riemann zeta function.

\begin{theorem}\label{thm1} 
As $n\to \infty$, we have the asymptotic relation 
$$p(n)\  \sim\  \nu(n)\sqrt{\frac{n}{\zeta(2)}}.$$
 \end{theorem}

\begin{proof}
We use the asymptotic estimate $1-e^{-x}\sim x$ as $x\to 0$, that arises from the Maclaurin series expansion for $f(x)=e^{-x}$.\footnote{These proofs for Theorems \ref{thm1} and \ref{thm2} were suggested by G. E. Andrews, and simplify our original proofs.} Note by \eqref{nu_recursion} that as $n\to \infty$, 
\begin{flalign}\label{pfeq1}
\frac{p(n)}{\nu(n)}\  &=\  \frac{p(n)}{p(n)-p(n-1)}\  =\  \left(1-\frac{p(n-1)}{p(n)}\right)^{-1} \sim\   \left(1-{ e^{A(\sqrt{n-1}-\sqrt{n})}}\right)^{-1}\\ \nonumber 
 &\sim\   \left(A\left(\sqrt{n}-\sqrt{n-1}\right)\right)^{-1}\  =\     A^{-1} \left(\sqrt{n}+\sqrt{n-1}\right)\  \sim\  2A^{-1}\sqrt{n},
\end{flalign} 
with $2A^{-1}=\sqrt{6/\pi^2}=\sqrt{\zeta(2)^{-1}}$, which leads to the statement of the theorem.  We use \eqref{H-R} for the first asymptotic above, use $1-e^{-x}\sim x$ for the second asymptotic, and use $\sqrt{n}+\sqrt{n-1}\sim 2\sqrt{n}$ for the third asymptotic after rationalizing the denominator.
 \end{proof}

\begin{remark}
We note that Theorem \ref{thm1} also follows from \eqref{H-R} together with \cite[Eq. (1.2)]{SillsChamp}.
\end{remark}
In fact, an even smaller subset of $\mathcal P$ can be shown to generate $\mathcal N$, and to control the growth of $\nu(n)$ in the same sense that $\nu(n)$ controls $p(n)$ \cite{non-unitary}. Let a {\it ground state non-unitary partition} denote a partition in $\mathcal N$ such that the largest part appears two or more times. 
Let $\mathcal G$ denote the set of all ground state non-unitary partitions, let $\mathcal G_n$ denote the ground state non-unitary partitions of size $n$, and let $\gamma(n)$ denote the cardinality $\#\mathcal G_n$ with $\gamma(0):=0$. One can compute $p(n)$ as a linear combination of the values $\gamma(k), k\leq n$ \cite[Eq. 5]{non-unitary}. Table \ref{Table} gives a comparison of the values of $\gamma(n), \nu(n), p(n)$ for small $n$. 
\begin{remark}
Consideration of Young diagrams shows the sets $\mathcal G$ and $\mathcal G_n$ are both closed under partition conjugation, which is also true for $\mathcal P$ and $\mathcal P_n$, but is not true for $\mathcal N$ or $\mathcal N_n$.\end{remark}
Note that a partition in $\mathcal G_n$ is a non-unitary partition of size $n$, and each element of $\mathcal N_n$ that is not an element of $\mathcal G_n$ can be constructed by adding 1 to the largest part of a partition in $\mathcal N_{n-1}$. Then we also have the recursion
\begin{equation}\label{gamma_recursion}
\nu(n)\  =\  \nu(n-1)+\gamma(n).
\end{equation} 
This recursion together with the given initial values yields $\nu(n)=1+\gamma(1)+\gamma(2)+\gamma(3)+\gamma(4)+\cdots+\gamma(n)$. One can deduce by repeated application of \eqref{H-R} that $\gamma(n)=o(\nu(n))$. 

Another approximation is conjectured in \cite{non-unitary} based on numerical evidence in Table \ref{Table}: 
\vskip.1cm

\begin{center}{{\it As $n\to \infty$, we have 
$p(2n)\  \approx\  2n\cdot\gamma(2n)$.}}\end{center}

\vskip.1cm
\noindent 
Again, the stated approximation is incompatible with \eqref{H-R}. However, with a multiplicative constant it turns out to be true, 
and extends to odd $n$, as a corollary of the next theorem. 

\begin{theorem}\label{thm2}  
As $n\to \infty$, we have the asymptotic relation 
$$\nu(n)\  \sim\  \gamma(n)\sqrt{\frac{n}{\zeta(2)}}.$$
 \end{theorem}

\begin{proof} 
Note by \eqref{gamma_recursion} that as $n\to \infty$, we have
\begin{flalign}\label{pfeq}
\frac{\nu(n)}{\gamma(n)}\  &=\  \frac{\nu(n)}{\nu(n)-\nu(n-1)}\  =\  \left(1-\frac{\nu(n-1)}{\nu(n)}\right)^{-1}   \\ \nonumber &\sim\  \left(1-\frac{p(n-1)\sqrt{n}}{p(n)\sqrt{n-1}}\right)^{-1} 
\nonumber \sim\  \left(1-\frac{p(n-1)}{p(n)}\right)^{-1}  \sim\  \sqrt{n/\zeta(2)},
\end{flalign} 
where we use Theorem \ref{thm1} for the first asymptotic, use $\sqrt{n/(n-1)}\sim 1$ for the second asymptotic, and refer back to \eqref{pfeq1} for the third asymptotic. 
\end{proof}

\begin{corollary} \label{cor1} 
As $n\to \infty$, we have the asymptotic relation
$$p(n)\  \sim\   \frac{n\cdot \gamma(n)}{\zeta(2)}.$$
 \end{corollary}

\begin{proof} Substitute the right side of the formula in Theorem \ref{thm2}, for $\nu(n)$ in Theorem \ref{thm1}.
\end{proof}

%
%
%
%

\begin{table}
\vskip.1in  \begin{center}
\begin{tabular}{ | c | c | c | c | }
\hline
{\bf $n$} &{\bf $\gamma(n)$} & {\bf $\nu(n)$}&  {\bf $p(n)$}\\ \hline
1 & 0 & 0 & 1   \\ \hline
2 & 0 & 1 &  2   \\ \hline
3 & 0 & 1 & 3   \\ \hline
4 & 1 & 2 & 5  \\ \hline
5 & 0 & 2 & 7  \\ \hline
6 & 2 & 4 & 11   \\ \hline
7 & 0 & 4 & 15   \\ \hline
8 & 3 & 7 & 22 \\  \hline
9 & 1 & 8 & 30   \\  \hline
10 & 4 & 12 & 42  \\  \hline
11 & 2 & 14 & 56  \\  \hline
12 & 7 & 21 & 77  \\  \hline
13 & 3 & 24 & 101  \\  \hline
14 & 10 & 34 & 135  \\  \hline
15 & 7 & 41 & 176  \\  \hline
16 & 14 & 55 & 231  \\  \hline
17 & 11 & 66 & 297  \\  \hline
18 & 22 & 88 & 385  \\  \hline
19 & 17 & 105 & 490  \\  \hline
20 & 32 & 137 & 627  \\  \hline
$\cdots$ & $\cdots$ & $\cdots$ & $\cdots$ \\  \hline
$100$ & \  $2{,}307{,}678$\   & \  $21{,}339{,}417$\   & \  $190{,}569{,}292$\  \\
\hline
\end{tabular}
\smallskip
\smallskip
\smallskip
\caption{Comparison of $\gamma(n), \nu(n), p(n)$ reproduced from \cite{non-unitary}.}\label{Table}\end{center}
\end{table}

 \section{Discussion of other conjectures and observations} 
\subsection{Questions of monotonicity of $\gamma(n)$}
In the previous section, we addressed conjectures raised in \cite{non-unitary} related to the asymptotic behaviors of $\nu(n)$ and $\gamma(n)$ as $n$ increases. There are other questions brought up in that work based on inspection of Table \ref{Table}, related to arithmetic properties of the sequences $\nu(n), \gamma(n)$, which we address in this section. 

As noted in \cite{non-unitary}, an immediate contrast between the $\gamma(n)$ column of Table \ref{Table} and the other two columns, is that while $p(n), \nu(n)$ both grow (weakly) monotonically,\footnote{That both $p(n)-p(n-1)=\nu(n)$ and $\nu(n)-\nu(n-1)=\gamma(n)$ are non-negative is clear.} the values of $\gamma(n)$ oscillate as they increase, with both $\gamma(2k-1)<\gamma(2k)$ and $\gamma(2k)>\gamma(2k+1)$ holding for the entries in the table. On the other hand, inspection of the table up to $n=20$  suggests $\gamma(n)$ does grow  monotonically if $n\to \infty$ through only odd or even $n$-values; this trend continues in our computations up to $n=3000$, and is readily proved.

\begin{proposition}\label{mono2}
For all $n\geq 3$, we have that $$\gamma(n)\geq \gamma(n-2).$$
\end{proposition}

\begin{proof}
For $n\geq 3$, adjoin $2$ as a part to every partition in $\mathcal G_{n-2}$, to produce the  partitions in $\mathcal G_{n}$ having smallest part $2$; thus $\gamma(n)$ is at least equal to $\gamma(n-2)$. Noting for $n\geq 6$ that $\mathcal G_{n}$ might also contain partitions with smallest part $\geq 3$, gives the inequality.  \end{proof}

On the other hand, our computations for $20<n\leq3000$ display another trend, contradicting our comments above about oscillating behavior. 

\begin{proposition}\label{conj}
For all $n\geq 26$, we have that $$\gamma(n)\geq\gamma(n-1).$$
\end{proposition}

\begin{proof}
For $f\colon \mathbb Z
\to \mathbb C$, let $\Delta f(n):=f(n)-f(n-1)$ denote the first difference of $f$, let $\Delta^2 f(n):=\Delta(\Delta f(n))$ denote the second difference, and for $r\geq 1$, let $\Delta^r f(n):=\Delta(\Delta^{r-1} f(n))$ denote the $r$th difference of the function $f$. Noting for $n\geq 2$ that $\nu(n)=\Delta p(n)$ and $\gamma(n)=\Delta^2 p(n)$, then the proposition is equivalent to the statement that $\Delta^3 p(n)\geq 0$ for all $n\geq 26$. In \cite{Gupta}, Gupta establishes that for each $r\geq 1$, there exists $n_0=n_0(r)\geq 1$ such that $\Delta^r p(n)\geq 0$ for all $n\geq n_0$. Gupta's proof exploits the Hardy-Ramanujan asymptotic \eqref{H-R} extended to $\Delta^r p(n)$, using an analytic argument about the error term; for $r=3$, Gupta gives $n_0(3)=26$.\footnote{The authors are grateful to F. Zanello for pointing out Gupta's proof to us.}  We note that Odlyzko proves in \cite{Odlyzko} that $\Delta^r p(n)$ oscillates for $n<n_0$, $r\geq 1$, and that $n_0(r)\sim \frac{6}{\pi^2}n^2 (\log r)^2$ as $n\to\infty$.  
\end{proof}

Table \ref{Table2} gives the values of $\gamma(n)$ and  $\gamma(n)-\gamma(n-1)$ for $20<n\leq 40$.  
%
Gupta's analytic proof \cite{Gupta} of Proposition \ref{conj}, while elegant, does not provide combinatorial insight into the behavior of $\gamma(n)$. Let us briefly analyze the problem further, by decomposing $\mathcal G_n$ as a union of two disjoint subsets. Let $\mathcal G_n^{(1)}\subseteq \mathcal G_n$ denote the subset of partitions in $\mathcal G_n$ such that 1 can be subtracted from some part with the resulting partition being in $\mathcal G_{n-1}$. Let $\mathcal G_n^{(2)}\subseteq \mathcal G_n$ denote size-$n$ partitions of either shape $(c,c)$ or $(c,c,2,2,\dots, 2)$, with $c\geq 2$, noting  $\mathcal G_n^{(2)}$  is empty if $n$ is odd, and  $\mathcal G_n^{(1)}\cap\mathcal G_n^{(2)}$ is empty.  
If a partition $\lambda\in \mathcal G_n$ does not lie in $\mathcal G_n^{(2)}$, then $\lambda$ will have at least one  part from which 1 can be subtracted with the result being in $\mathcal G_{n-1}$,  and lies in $\mathcal G_n^{(1)}$. Thus $\mathcal G_n = \mathcal G_n^{(1)}\cup \mathcal G_n^{(2)}$, and   \begin{equation}\label{gammasum} \gamma(n)  =  \#\mathcal G_n^{(1)}+ \#\mathcal G_n^{(2)}.\end{equation}

We derive formulas for $\#\mathcal G_n^{(1)}, \#\mathcal G_n^{(2)}$. For $x\in\mathbb R$,  define a {modified floor function}  $\left\lfloor x \right\rfloor^*:= 0$ if $x$ is an integer, and $\left\lfloor x \right\rfloor^*:=\left\lfloor x\right\rfloor$ otherwise. 
Each partition $\pi\in\mathcal G_{2k}^{(2)}, k>1,$ has either the form $\pi=
(k-j, k-j, 2, 2, \dots, 2)$ with $k-j>2$ and $j$ copies of 2, $0\leq j < k-2$, or $\pi=(2,2,\dots,2)$ with $k$ copies of 2. Then  
$\#\mathcal G_n^{(2)}=\left\lfloor \frac{n-1}{2} \right\rfloor^*.$ 

To address $\#\mathcal G_n^{(1)}$, we need to define another, closely-related subset. For $k\geq 1$, let $\mathcal G_k^{(0)}\subseteq\mathcal G_k$ denote the subset of  partitions in $\mathcal G_k$ that are {\it not} of the form  $(d,d,\dots,d)$ with $d\in \mathbb N$ a nontrivial divisor of $k$, i.e., those $\sigma_0(k)-2$ partitions in $\mathcal G_k$ whose Young diagrams are  not rectangles, with $\sigma_0(k)$ the number of divisors of $k$. 
It follows that \begin{equation}\label{G0} \#\mathcal G_k^{(0)}=\gamma(k)-\sigma_0(k)+2.\end{equation}

Now, adding 1 to an allowable part of  a partition  $\lambda$ in $\mathcal G_{n-1}^{(0)}$ (allowable, in that adding 1 results in another ground state non-unitary partition) yields a partition in $\mathcal G_n^{(1)}$, and this can be done in one or more ways for each $\lambda\in \mathcal G_{n-1}^{(0)}$. This resembles the construction of the partitions $\mathcal P_n$ from the partitions $\mathcal P_{n-1}$ in Young's lattice in representation theory  \cite{Young}.\footnote{The authors are grateful to J. Lagarias for a detailed discussion of the construction of Young's lattice.} 
  Likewise, subtracting 1 from an allowable part of a partition $\lambda'$ in $\mathcal G_n^{(1)}$ (allowable, in that subtracting 1 produces a ground state non-unitary partition) yields a partition in $\mathcal G_{n-1}^{(0)}$, and this also might be done in multiple ways for each $\lambda'\in \mathcal G_n^{(1)}$.

Due to multiple copies of some partitions being produced in both directions, leading to potential over-counting, it is a difficult combinatorial task to compare the cardinalities of $\mathcal G_{n-1}^{(0)}$ and $\mathcal G_n^{(1)}$ by keeping track of these ``add-one'' and ``subtract-one'' mappings. 
Define a statistic $\varepsilon(n)\in \mathbb Z$ to track the difference between these cardinalities:
\begin{equation}\label{epsilondef} \#\mathcal G_{n}^{(1)}= \#\mathcal G_{n-1}^{(0)}+\varepsilon(n).\end{equation}  
Rewriting  \eqref{gammasum} in light of \eqref{G0} and \eqref{epsilondef}, we arrive at an identity 
\begin{equation}\label{gamma*}\gamma(n)\  =\  \gamma(n-1)+\varepsilon(n)-\sigma_0(n-1)+\left\lfloor \frac{n-1}{2} \right\rfloor^* + 2,\end{equation}  
somewhat analogous to \eqref{nu_recursion} and \eqref{gamma_recursion}, that identifies the difference between $\gamma(n)$ and $\gamma(n-1)$. From \eqref{gamma*}, we see $\varepsilon(n)$ captures the ``wild card'' component in the growth of $\gamma(n)$.  
Then it follows from Proposition \ref{conj} that,  
for $n\geq 26$, we have  $$\varepsilon(n)\  \geq\  \sigma_0(n-1)-\left\lfloor \frac{n-1}{2} \right\rfloor^* -2. $$

\begin{table}
\vskip.1in  \begin{center}
\begin{tabular}{ | c | c | c | c | }
\hline
{\bf $n$} &{\bf $\gamma(n)$} & {\bf $\gamma(n)-\gamma(n-1)$}\\ \hline
21 & 28 & $-4$   \\ \hline
22 & 45 & 17    \\ \hline
23 & 43 & $-2$    \\ \hline
24 & 67 & 24  \\ \hline
25 & 63 & $-4$   \\ \hline
26 & 95& 32    \\ \hline
27 & 96& 1    \\ \hline
28 & 134 & 38  \\  \hline
29 & 139 & 5   \\  \hline
30 & 192 & 53   \\  \hline
31 & 199 & 7  \\  \hline
32 & 269 & 70   \\  \hline
33 & 287 & 18  \\  \hline
34 & 373 & 86   \\  \hline
35 & 406 & 33 \\  \hline
36 & 521 & 115   \\  \hline
37 & 566 & 45  \\  \hline
38 & 718 & 152   \\  \hline
39 & 792 & 74   \\  \hline
40 & 983 & 191   \\  \hline
\end{tabular}
\smallskip
\smallskip
\smallskip
\caption{Growth of $\gamma(n)$ is monotonic for $n\geq 25$}\label{Table2}\end{center}
\end{table}

\subsection{Congruences and common divisors}\label{sect2}
For $n\geq 0$, recall the Ramanujan congruences for the partition function \cite{Ramanujan}:
\begin{flalign}\label{Ramcong}
p(5n+4)\equiv 0\  (\operatorname{mod} 5),\  \  \  \  
p(7n+5)\equiv 0\  (\operatorname{mod} 7),\  \  \  \   
p(11n+6)\equiv 0\  (\operatorname{mod} 11).
\end{flalign}
In \cite{non-unitary}, the sixth author made a conjecture based on numerical evidence in Table \ref{Table}:  
{\it Congruences similar to \eqref{Ramcong} exist for $\nu(n)$ and $\gamma(n)$.} 
However, potential patterns suggestive of this conjecture were likely coincidences based on the small data set in Table \ref{Table}. In our computations up to $n=3000$, we did not spot Ramanujan congruences for $\nu(5n+4), \gamma(7n+5)$, etc.;\footnote{We did not check for more general types of congruences such as in \cite{AO}, for $\nu(an+b), \gamma(an+b)$.}  but we spotted other 
Ramanujan-like congruences involving $\nu(n), \gamma(n),$ that we then proved. We record these observed congruences and their proofs. 

\begin{proposition} \label{thm3}
For $n\geq 1$ we have the following:
\begin{flalign*}\label{Ramcong}
p(5n)&\equiv \nu(5n)\  (\operatorname{mod} 5),\\ \nonumber
p(7n+6)&\equiv \nu(7n+6)\  (\operatorname{mod} 7),\\ \nonumber
p(11n+7)&\equiv \nu(11n+7)\  (\operatorname{mod} 11).
\end{flalign*}
%
%
%
 \end{proposition}

\begin{proof} Recall $a\  |\  b$ means $a\in \mathbb N$ divides $b\in \mathbb N$. The claimed relations follow immediately from the Ramanujan congruences \eqref{Ramcong} combined with \eqref{nu_recursion}, since the statement $p(m)\equiv 0 \  (\operatorname{mod} c)$ is equivalent to the statement $c\  |\  \left(p(m+1)-\nu(m+1)\right)$. 
\end{proof}

The proof above depends  on known congruences for $p(n)$; without comparable results for $\nu(n)$, 
we cannot prove analogous results involving $\gamma(n)$. 
Other congruences $p(an+b)\equiv 0\  (\text{mod}\   c)$ (see  \cite{AO}) will yield further congruences $p(an+b+1)\equiv \nu(an+b+1)\  (\text{mod}\   c)$. 

In our computer searches up to $n=3000$, we also noticed interesting $\text{gcd}$ relations; we were initially surprised to find that $p(n),\nu(n)$ have nontrivial greatest common divisor  in $94.6\%$ of cases. Moreover, in $91.8\%$ of cases, all three of $p(n), \nu(n), \gamma(n)$ have $\operatorname{gcd}>1$. 
Following up on these empirical observations, we proved by elementary means the divisibility properties of $p(n),\nu(n),\gamma(n)$ are not independent. 

\begin{proposition} \label{congruence_thm2}
For $n\geq 1$, we have the following: 
\begin{enumerate}[(i)]
\item $ \operatorname{gcd}\left(p(n), \nu(n)\right)\ =\ \operatorname{gcd}\left(\nu(n), p(n-1)\right)\ =\ \operatorname{gcd}\left(p(n), p(n-1)\right);$
\vskip.1cm
\item $\operatorname{gcd}\left(\nu(n), \gamma(n)\right)\  =\  \operatorname{gcd}\left(\gamma(n), \nu(n-1)\right)\  =\  \operatorname{gcd}\left(\nu(n), \nu(n-1)\right).$
 \end{enumerate} \end{proposition}

\begin{proof}
Both of the congruence relations in the theorem are instances of the following fact: if $x=y+z$ for $x,y,z\in \mathbb N$, then 
$\operatorname{gcd}\left(x, y\right)\ =\ \operatorname{gcd}\left(y, z\right)\ =\ \operatorname{gcd}\left(x, z\right)$. 
To see this, set $d:=\operatorname{gcd}\left(x,y\right)$; then $d$ is also a divisor of $z=x-y$. If $d\geq 1$ does {\it not} equal $\operatorname{gcd}\left(y,z\right)$, then there must exist $d'>d$ such that $d'$ divides both $y,z$. But then $d'$ also divides $x=y+z$, contradicting that $d=\operatorname{gcd}\left(x,y\right)$. This proves that $d=\operatorname{gcd}\left(y,z\right).$ 
By a similar argument, if there exists $d''>d$ that divides both $x,z$, then $d''$ divides $y=x-z$, again contradicting that $d=\operatorname{gcd}\left(x,y\right)$. This proves $d=\operatorname{gcd}\left(x,z\right).$ Setting  $(x,y,z)=\left(p(n), \nu(n), p(n-1)\right)$ gives (i). Setting $(x,y,z)=\left(\nu(n), \gamma(n), \nu(n-1)\right)$ gives (ii).
\end{proof}

It is interesting that congruence and divisibility properties of $p(n)$ depend to some extent on the more primitive functions $\nu(n), \gamma(n)$. In \cite{non-unitary}, it is conjectured these phenomena could be used to reverse-engineer a proof of the Ramanujan congruences \eqref{Ramcong} by induction, which would be a useful application, if it is possible. 
We note that, while the proofs of Propositions  \ref{thm3} and \ref{congruence_thm2} explain our observations, in those cases they do not indicate 
 new partition-theoretic phenomena.  On the other hand, our group found $\operatorname{gcd}$ relations we could not prove, connecting $p(n), \nu(n),$ and $\gamma(n)$, to be ubiquitous in our numerical data.
Further study of arithmetic connections between these functions seems warranted.

\section{Non-unitary partition connections in the authors' other works}\label{sect3}
 
Beyond the earlier usages \cite{N1, Guy, N2, N3} of non-unitary partitions in the literature, and the study in \cite{non-unitary}, non-unitary partitions have distinguished themselves in the present authors'  works as being of interest. In \cite{SillsChamp}, non-unitary partitions
arise in a statistical application, in the context of estimating a population's standard deviation using a linear combination of ranges of subsamples of a sample of size $n$, which are indexed by partitions of size $n$; only non-unitary partitions can be used (they are called  ``admissible'' partitions there).  
In the theory of partition zeta functions \cite{Robert_zeta}, the partition-theoretic zeta series diverge over subsets of $\mathcal P$ where 1's can appear as parts with unbounded multiplicity, and non-unitary partitions are a necessary condition for the existence of partition Euler products. In \cite[Thm. 29]{SS_norm}, the Dirichlet series generating function for the number of non-unitary partitions with {norm} equal to $n$ is evaluated.  In \cite{Robert_bracket}, only partitions with no 1's have nonzero partition phi function $\varphi_{\mathcal P}(\lambda)$. In  \cite{DawseyJustSchneider}, under the supernorm isomorphism, the set $\mathcal N$ maps bijectively to the odd natural numbers, while partitions with $1$'s map to the even numbers.  
Do non-unitary partitions distinguish themselves naturally in other contexts in the mathematical sciences?

\section*{Acknowledgments}
We are very grateful to George Andrews for pointing out simplified proofs of Theorems \ref{thm1} and \ref{thm2}, and for providing references to the classical theory of invariants; to Jeffrey Lagarias for advice about Young's lattice and injective maps on partitions; to Fabrizio Zanello for providing the Gupta reference on finite differences of $p(n)$; and 
to Sophia Ramirez for making additional computations as an  early member of our group. We are also thankful to the anonymous referee, for suggestions that improved the paper substantially. 

This work stems from the cross-institutional Computational and Experimental Number Theory group sponsored jointly by M. Hendon (University of Georgia), R. Schneider (Michigan Technological University), and A. V. Sills (Georgia Southern University), involving graduate students, undergraduates, and faculty collaborators. The group used Python, Mathematica, Magma, Google Sheets, Microsoft Excel, and Texas Instruments calculators for computations.
Partial support for the first and second authors was provided by the Research and Training Group grant DMS-1344994 funded by the National Science Foundation.
The second author was also supported in part by the National Science Foundation Graduate Research Fellowship under Grant No. 1842396. The third author was partially supported by a research assistantship through University of Georgia's Center for Undergraduate Research Opportunities.

\end{document}